\documentclass[11pt]{article}
\usepackage[english]{babel}
\usepackage{amssymb,amsmath,amsthm}
\newtheorem{theorem}{Theorem}[section]
\newtheorem{lemma}[theorem]{Lemma}
\newtheorem{proposition}[theorem]{Proposition}
\newtheorem{corollary}[theorem]{Corollary}

\newtheorem{question}[theorem]{Question}

{\theoremstyle{definition}}
{\theoremstyle{definition}\newtheorem{definition}[theorem]{Definition}}
{\theoremstyle{definition}\newtheorem{remark}[theorem]{Remark}}

\numberwithin{equation}{section}

\def\N{{\mathbb N}}

\def\K{{\mathbb K}}

\def\epsilon{\varepsilon}
\def\kappa{\varkappa}
\def\phi{\varphi}
\def\leq{\leqslant}
\def\geq{\geqslant}
\def\supp{\hbox{\tt supp}\,}
\def\dim{{\rm dim}\,}

\def\supp{\hbox{\tt supp}\,}

\title{Finite dimensional semigroup quadratic algebras\\ with the minimal number of relations}

\author{Natalia Iyudu and Stanislav Shkarin}

\date{}

\begin{document}

\maketitle

\begin{abstract}A quadratic semigroup algebra is an algebra over a field given by
the generators $x_1,\dots,x_n$ and a finite set of quadratic
relations each of which either has the shape $x_jx_k=0$ or the shape
$x_jx_k=x_lx_m$. We prove that a quadratic semigroup algebra given
by $n$ generators and $d\leq \frac{n^2+n}{4}$ relations is always
infinite dimensional. This strengthens the Golod--Shafarevich
estimate for the above class of algebras. Our main result however is
that for every $n$, there is a finite dimensional quadratic
semigroup algebra with $n$ generators and $\delta_n$ relations,
where $\delta_n$ is the first integer greater than
$\frac{n^2+n}{4}$. That is, the above Golod--Shafarevich-type estimate for semigroup algebras is sharp.
\end{abstract}



\section{Introduction \label{s1}}\rm

Throughout the paper, $\K$ is an arbitrary field and $\N$ is the
set of positive integers. For a set $X$, $\K\langle X\rangle$ stands
for the free associative algebra over $\K$ generated by $X$. For
$n\in\N$, $\langle X\rangle_n$ denotes the set of monomials in
$\K\langle X\rangle$ of degree $n$: $\langle
X\rangle_n=\{x_{i_1}\dots x_{i_n}: x_{i_j} \in X \}$.

We deal with {\it quadratic algebras}, that is, algebras $R$ given
as $\K\langle X\rangle/I$, where $I$ is the ideal in $\K\langle
X\rangle$ generated by a collection of homogeneous elements (called
relations) of degree $2$.

Algebras of this class, their growth, their Hilbert series, rationality and nil/nilpotency properties have been extensively studied (see \cite{popo, Ag, ufn} and references therein). One of the most challenging  questions in the area  (see the ICM paper \cite{Ag}, or \cite{z1}) is whether there exists an infinite dimensional nil algebra in this class. This is a version of the Kurosh problem with an additional constraint  formulated in terms of defining relations. A version of the Kurosh problem with a restriction on the rate of growth was recently solved in \cite{AgL}.
As it is well-known, the original Kurosh problem was solved with the help of  the Golod--Shafarevich lower estimate for the Hilbert series of an algebra. This shows that the behavior of the Hilbert series plays a crucial role in the structural theory of algebras. There is a vast literature devoted to the better understanding of the Golod--Shafarevich inequality. For instance, in \cite{Sch,cana} a number of results on asymptotic (in various senses) tightness of the Golod--Shafarevich inequality is obtained. In \cite{cana} we also prove some partial non-asymptotic results results on the Anick conjecture.

The Golod--Shafarevich-type estimates have very deep connections to other areas of algebra. First, it is necessary to mention their classical applications to $p$-groups and  class field theory \cite{gosh,goshp,z2}. As evidence of the connection to certain deep homological properties one can consider the fact that
for quadratic algebras with $n$ generators and $d$ relations, the interval $\frac{n^2}{4}< d < \frac{n^2}{2}$, where the Anick conjecture
on the tightness of the Golod--Shafarevich estimate is difficult to tackle, is exactly the interval, for which the generic algebra is non-Koszul (see \cite{popo}). Furthermore the notion of noncommutative complete intersection \cite{eti} can be defined in terms of the Hilbert series and refers to a class of algebras for which the Golod--Shafarevich inequality turns into equality. On the other hand, this concept is also characterized by the nice behavior of the noncommutative Koszul (Shafarevich) complex.

In this paper, we mostly deal with quadratic semigroup algebras, which traditionally serve as a main source of examples in this area and where an interesting combinatorics of words can applied. This class was extensively studied, one can find an account of main results and methods in \cite{J}.

The Golod--Shafarevich lower estimate \cite{gosh,popo} for the Hilbert series of an algebra given by quadratic relations implies \cite{Vinb} that $R$ is infinite dimensional if $d\leq \frac{n^2}{4}$, where $R$ is an algebra defined by $n$
generators and $d$ quadratic relations. Anick \cite{ani1,ani2}
conjectured that the Golod--Shafarevich estimate is attained in the
case of quadratic algebras. In particular, this conjecture, if true,
implies that for all $d,n\in\N$ with $d>\frac{n^2}{4}$, there is a
finite dimensional quadratic algebra with $n$ generators defined by
$d$ relations. This conjecture in its full generality still remains open.

We study the same question for a subclass of the class of quadratic
algebras.

\begin{definition}\label{semiq} A {\it quadratic semigroup algebra}
is an algebra given by generators and relations such that each of
its defining relations is either a degree 2 monomial or a difference
of two degree 2 monomials.
\end{definition}

Many partial solutions of Anick's conjecture were obtained by means of examples, which are quadratic semigroup algebras. An interesting combinatorial result dealing with semigroup relations was obtained by Wisliceny. We are moving in the same direction.
Wisliceny \cite{wis} proved that the conjecture of Anick is
asymptotically correct. Namely, for every $n\in\N$, he constructed a
quadratic algebra $R_n$ with $n$ generators defined by $d_n$
relations such that $R_n$ is finite dimensional and
$\lim\limits_{n\to\infty} \frac{d_n}{n^2}=\frac14$. More
specifically, $d_n=\frac{n^2+2n}{4}$ if $n$ is even and
$d_n=\frac{n^2+2n+1}{4}$ if $n$ is odd. The
algebras $R_n$, constructed by Wisliceny, are quadratic semigroup
algebras.

It turns out that it is not enough to consider quadratic semigroup algebras to prove  Anick's conjecture. Namely, the Golod--Shafarevich estimate can be
improved for this class.

\begin{theorem}\label{semi1} Let $R$ be a quadratic semigroup algebra
with $n$ generators given by $d\leq\frac{n^2+n}{4}$ relations. Then
$R$ is infinite dimensional.
\end{theorem}

However, this leaves a gap between the $\frac{n^2}{4}+\frac{n}{4}$
of Theorem~\ref{semi1} and approximately $\frac{n^2}{4}+\frac{n}{2}$ of the
Wisliceny example. Surprisingly, it turns out that the estimate
provided by Theorem~\ref{semi1} is tight.

\begin{theorem}\label{semi2} Let $d,n\in\N$ be such that
$d>\frac{n^2+n}{4}$. Then there exists a finite dimensional
quadratic semigroup algebra with $n$ generators given by $d$
relations.
\end{theorem}

Thus the minimal possible number of defining relations of a finite
dimensional quadratic semigroup algebra with $n$ generators is
precisely the first integer greater than $\frac{n^2+n}{4}$.
Theorems~\ref{semi1} and~\ref{semi2} give a solution
to the natural analog of Anick's conjecture for the class of semigroup
quadratic algebras. Furthermore, it roughly halves the gap between
the number of relations in the original Anick's conjecture and
the number of relations in Wisliceny's example, which
gave the best bound known before the result of this paper.

\section{Proof of Theorem~\ref{semi1}}

\begin{definition}\label{support} Let $g=\sum\limits_{a,b\in
X}\lambda_{a,b}ab$ be a homogeneous degree $2$ element of $\K\langle
X\rangle$. Then the finite set $\{ab:\lambda_{a,b}\neq 0\}\subseteq
\langle X\rangle_2$ is called {\it the support of} $g$ and is
denoted $\supp(g)$.
\end{definition}

\begin{lemma}\label{SE} Let $X$ be a non-empty set,
$M$ be a family of homogeneous degree $2$ elements of $\K\langle
X\rangle$ and $R=\K\langle X\rangle/I$, where $I$ is the ideal
generated by $M$. Assume also that there are $a,b\in X$ such that
$ab,ba\notin\supp(f)$ for every $f\in M$. Then $R$ is infinite
dimensional.
\end{lemma}

\begin{proof} Since $ab$ and $ba$ are not in the support of any
$f\in M$, the monomial $(ab)^n$ does not feature in the polynomial
$pfq$ for every $f\in M$ and $p,q\in \K\langle X\rangle$. Hence
$(ab)^n\notin I$ for each $n\in\N$. It immediately follows that $R$
is infinite dimensional.
\end{proof}

\begin{lemma}\label{SE1} Let $X$ be a non-empty set,
$M$ be a family of homogeneous degree $2$ elements of $\K\langle
X\rangle$ and $R=\K\langle X\rangle/I$, where $I$ is the ideal
generated by $M$. Assume also that for each $f\in M$, the
coefficients of $f$ sum up to $0$. Then $R$ is infinite dimensional.
\end{lemma}

\begin{proof} Since the coefficients of each $f$ sum up to
$0$, it follows that the coefficients of every element of $I$ sum up
to $0$. Hence $I$ contains no monomials and therefore $R$ is
infinite dimensional.
\end{proof}

We are ready to prove Theorem~\ref{semi1}. Denote by $X$ the
$n$-element set of generators of $R$. If there are $a,b\in X$ such
that $ab$ and $ba$ are both not in the support of any relation, then
$R$ is infinite dimensional according to Lemma~\ref{SE}. Thus we can
assume that for every $a,b\in X$ either $ab$ or $ba$ (or both)
belongs to the union $N$ of the supports of the relations. It
follows that $|N|\geq\frac{n^2+n}{2}$. Since each relation is either
a monomial or a difference of two monomials, their supports have at
most two elements. Taking into account that we have
$d\leq\frac{n^2+n}{4}$ relations, we are left with the only
possibility that $d=\frac{n^2+n}{4}$ and each relation is a
difference of two monomials. In this case, the coefficients of each
relation sum up to 0. By Lemma~\ref{SE1}, $R$ is infinite
dimensional. The proof of Theorem~\ref{semi1} is complete.

\section{Quasierh\"ohungssysteme}

We find the algebras from Theorem~\ref{semi2} within the following
class considered by Wisliceny \cite{wis}.

\begin{definition}\label{qhs} Let $X$ be a finite set of generators
carrying a total ordering $<$. A set $M\subset \K\langle X\rangle$
of homogeneous elements of degree 2 is called a {\it Quasierh\"ohungssysteme  $($QHS\,$)$} on $X$ if
$$
\text{$\bigcup\limits_{g\in M}\supp(g)=\{ab:a,b\in X,\ a\geq b\}$
and the sets $\supp(g)$ are pairwise disjoint};
$$
and for every $g\in M$ one of the following three possibilities
holds:
\begin{align*}
&\text{either $g=ab$ with $a,b\in X$, $a\geq b$};
\\
&\text{or $g=ab-cd$ with $a,b,c,d\in X$, $a\geq b>c\geq d$};
\\
&\text{or $g=ab-cd$ with $a,b,c,d\in X$, $a>b=c>d$}.
\end{align*}
\end{definition}

Each QHS $M$ generates an ideal $J_M$ in $\K\langle
X\rangle$ and $R_M=\K\langle X\rangle/J_M$ is a
quadratic semigroup algebra. It is easy to see that the
cardinalities of QHSs on an $n$-element set $X$
range from the first integer $\delta_n$ greater than
$\frac{n^2+n}{4}$ to $\frac{n^2+n}{2}$. In the subsequent sections
we shall prove Theorem~\ref{semi2} by showing that there is a QHS
$M$ on an $n$-element $X$ of cardinality $\delta_n$ such that the
corresponding quadratic semigroup algebra $R_{M}$ is finite
dimensional.

This class of algebras was introduced by Wisliceny \cite{wis} and his example is from
the same class. He also provided an example of a QHS $M$ on a
$6$-element $X$ for which the corresponding algebra is infinite
dimensional. There is a criterion in \cite{wis} for a QHS to produce
a finite dimensional algebra, but although it can be handy for
treating specific examples with small number of generators with an
aid of a computer, it is nearly impossible to use when $X$ has a large number of
elements. Our proof does not make use of this criterion.

\begin{remark}\label{rem1} Let $X$ be a finite totally ordered set
and $M$ be a QHS on $X$. Let also $a$ be the minimal element of $X$
and $b$ be the maximal element of $X$. Form the definition of a
QHS it easily follows that the monomial $ba$ belongs
to $M$. We shall frequently deal with the set $M'=M\setminus \{ba\}$
and the ideal $I_M$ in $\K\langle X\rangle$ generated by $M'$.
\end{remark}

We will also use the following notation. If $I$ is an ideal in
$\K\langle X\rangle$ and $f,g\in \K\langle X\rangle$, the equality
$f=g\,(\bmod\,I)$ means that $f-g\in I$. Furthermore, if $X$ is
totally ordered and $m\in\N$, we shall {\bf always} endow the set of
monomials $\langle X\rangle_m$ of degree $m$ with the {\bf
right-to-left} lexicographical ordering. That is, $u_1\dots
u_m<v_1\dots v_m$ if there is $j\in\{1,\dots,m\}$ such that
$u_j<v_j$ and $u_l=v_l$ whenever $l>j$.

\begin{definition}\label{abab} Let $X$ be a finite totally ordered set,
$M$ be a QHS on $X$, $a=\min X$, $b=\max X$, $M'=M\setminus\{ba\}$
and $I_M$ be the ideal in $\K\langle X\rangle$ generated by $M'$. A
monomial $u=u_1\dots u_m\in\langle X\rangle_m$ is called {\it
minimal} (with respect to $M$) if $u\notin I_M$ and $u\leq v$
whenever $v\in\langle X\rangle_m$ and $u=v\,(\bmod\,I_M)$ (in other
words, $u$ is minimal in the right-to-left lexicographical
ordering among the degree $m$ monomials belonging to the coset
$u+I_M$).

We say that a minimal monomial $u\in\langle X\rangle_m$ is {\it
tame} if there is $v\in\langle X\rangle_m$ such that
$u=v\,(\bmod\,I_M)$ and there is $j\in\{1,\dots m\}$ such that
$v_j=b$ and $v_l=u_l$ for every $l>j$.

A non-tame minimal monomial will be called {\it singular}.

We say that the QHS $M$ is {\it regular} if there is
$m\in\N$ such that there are no singular monomials of degree $m$.
\end{definition}

\begin{remark}\label{rem2} It is easy to see that under the
assumptions of Definition~\ref{abab}, for any $v\in\langle
X\rangle_m$ such that $v\notin I_M$, there is a unique minimal
monomial $u\in\langle X\rangle_m$ satisfying $v=u\,(\bmod\,I_M)$.
Furthermore, if $v$ does not belong to the ideal $J_M$ in $\K\langle
X\rangle$ generated by $M$, then $u\notin J_M$ as well. The latter
happens because $I_M\subset J_M$.
\end{remark}

As usual, a {\it submonomial} of a monomial $u_1u_2\dots u_m$ is a
monomial of the shape $u_ju_{j+1}\dots u_k$, where $1\leq j\leq
k\leq m$. From the above definition it immediately follows that a
submonomial of a minimal monomial is minimal and that a minimal
monomial which has a tame submonomial is tame itself. This, in turn,
implies that a submonomial of a singular monomial is singular. These
observations are summarized in the next lemma.

\begin{lemma}\label{subm} Let $X$ be a finite totally ordered set
and $M$ be a QHS on $X$. Let also $u\in\langle X\rangle_m$  and $v$
be a submonomial of $u$. Then minimality of $u$ implies minimality
of $v$ and singularity of $u$ implies singularity of $v$.
\end{lemma}

The following lemma shows the relevance of the above definition.

\begin{lemma}\label{regqhs} Let $X$ be a finite totally ordered set and
$M$ be a regular QHS on $X$. Then the corresponding quadratic
semigroup algebra $R_M$ is finite dimensional.
\end{lemma}

\begin{proof}Let $a=\min X$, $b=\max X$, $M'=M\setminus\{ab\}$,
$I_M$ and $J_M$ be the ideals in $\K\langle X\rangle$, generated by
$M'$ and $M$ respectively. Since $M$ is regular, there is $m\in\N$
such that there are no singular monomials of degree $m$. In order to
show that $R_M=\K\langle X\rangle/J_M$ is finite dimensional it
suffices to show that $R_M$ is $(m+1)$-step nilpotent, or
equivalently, that $\langle X\rangle_{m+1}\subset J_M$. Assume the
contrary. Then (see Remark~\ref{rem2}) there is a minimal monomial
$u=u_1\dots u_mu_{m+1}\in \langle X\rangle_{m+1}$, which does not
belong to $J_M$. By Lemma~\ref{subm}, the submonomial $\widetilde
u=u_1\dots u_m$ is also minimal. Since there are no singular
monomials of degree $m$, $\widetilde u$ is tame. That is, there are
$j\in\{1,\dots,m\}$ and $\widetilde v=v_1\dots v_m\in \langle
X\rangle_{m}$ such that $\widetilde u=\widetilde v\,(\bmod\,I_M)$,
$v_j=b$ and $v_l=u_l$ whenever $j<l\leq m$. If
$v_ju_{j+1}=bu_{j+1}\in M$, then
$$
u=v_1\dots v_{j-1}bu_{j+1}\dots u_{m+1}=0\,(\bmod\,J_M)
$$
and therefore $u\in J_M$, which is a contradiction. If
$bu_{j+1}\notin M$, then since $b=\max X$, the definition of a QHS
implies that there are $c,d\in X$ such that
$bu_{j+1}-cd\in M'$ and $u_{j+1}>d$. Hence
$$
u=v_1\dots v_{j-1}bu_{j+1}\dots u_{m+1}=v_1\dots
v_{j-1}cdu_{j+2}\dots u_{m+1}\,(\bmod\,I_M).
$$
Since the last monomial in the above display is less than $u$ (in
the right-to-left lexicographical ordering) we have obtained a
contradiction with the minimality of $u$.
\end{proof}

We also need the following property of singular monomials.

\begin{lemma}\label{sing} Let $X$ be a finite totally ordered set,
$a=\min X$, $b=\max X$, $M$ be a QHS on $X$, $M'=M\setminus\{ba\}$,
$I_M$ be the ideal in $\K\langle X\rangle$ generated by $M'$ and
$u=u_1\dots u_m\in\langle X\rangle_m$ be a singular monomial. Assume
also that $1\leq k\leq m$ and $v=v_1\dots v_k\in\langle X\rangle_k$
is such that $v=u_1\dots u_k\,(\bmod\,I_M)$. Then there exists
$w=w_1\dots w_m\in\langle X\rangle_m$ such that $w_m\geq v_k$ and
$u=w\,(\bmod\,I_M)$.
\end{lemma}

\begin{proof} We shall use induction by $m-k$. If $m-k=0$, then
$m=k$ and $w=v$ satisfies all desired conditions. Thus we have our
basis of induction. Assume now that $n\in\N$ and that the statement
of the lemma holds whenever $m-k<n$. We have to prove it the case
$m-k=n$.

First, consider the case $u_{k+1}\geq v_k$. Since $v_1\dots
v_ku_{k+1}=u_1\dots u_ku_{k+1}\,(\bmod\,I_M)$ and the difference
between the degrees of $u$ and $v_1\dots v_{k}u_{k+1}$ is $n-1$, the
induction hypothesis provides us with $w=w_1\dots w_m\in\langle
X\rangle_m$ such that $w_m\geq u_{k+1}$ and $u=w\,(\bmod\,I_M)$.
Since $u_{k+1}\geq v_k$, we have $w_m\geq v_k$, as required.

It remains to consider the case $u_{k+1}<v_k$. Observe that $v_k\neq
b$. Indeed, otherwise the equality
\begin{equation}\label{equ2}
v_1\dots v_{k-1}v_ku_{k+1}\dots u_m=u\,(\bmod\,I_M)
\end{equation}
implies that $u$ is tame. Since $u_{k+1}<v_k$, the degree 2 monomial
$v_ku_{k+1}$ features in exactly one of the supports of the elements
of $M$ (see the definition of a QHS). Since $v_k\neq b$,
$v_ku_{k+1}\neq ba$ and therefore there is exactly one $g\in M'$
such that $v_ku_{k+1}\in\supp(g)$. There are three possibilities:
either $g=v_ku_{k+1}$ or $g=v_ku_{k+1}-cd$ with $d<u_{k+1}$ or
$g=cd-v_ku_{k+1}$ with $d\geq v_k$. If $g=v_ku_{k+1}$, we have
$v_ku_{k+1}\in I_M$ and therefore (\ref{equ2}) implies that $u\in
I_M$, which contradicts the minimality of $u$. Hence
$g=v_ku_{k+1}-cd$ with $d<u_{k+1}$ or $g=cd-v_ku_{k+1}$ with $d\geq
v_k$. In either case $v_ku_{k+1}=cd\,(\bmod\,I_M)$ and (\ref{equ2})
implies that
\begin{equation}\label{equ3}
u=v_1\dots v_{k-1}cdu_{k+2}\dots u_m\,(\bmod\,I_M).
\end{equation}
In the case $d<u_{k+1}$ the monomial in the right-hand-side of
(\ref{equ3}) is less than $u$, which contradicts the minimality of
$u$. It remains to consider the case $d\geq v_k$. Since the
difference between the degrees of $u$ and $v_1\dots v_{k-1}cd$ is
$n-1$, by the induction hypothesis, there is $w=w_1\dots
w_m\in\langle X\rangle_m$ such that $w_m\geq d$ and
$u=w\,(\bmod\,I_M)$. Since $d\geq v_k$, we have $w_m\geq v_k$, as
required.
\end{proof}

We also need a slight improvement of Lemma~\ref{sing} in a
particular case.

\begin{definition}\label{pure} Let $X$ be a finite totally ordered
set and $M$ be a QHS on $X$. We say that $d\in X$ is {\it pure} if
$a\geq b>c\geq d$ whenever $ab-cd\in M$.
\end{definition}

\begin{lemma}\label{singp} Let $X$ be a finite totally ordered set,
$a=\min X$, $b=\max X$, $M$ be a QHS on $X$, $M'=M\setminus\{ba\}$,
$I_M$ be the ideal in $\K\langle X\rangle$ generated by $M'$ and
$u=u_1\dots u_m\in\langle X\rangle_m$ be a singular monomial such
that $u_m$ is pure. Assume also that $1\leq k<m$ and $v=v_1\dots
v_k\in\langle X\rangle_k$ is such that $v=u_1\dots
u_k\,(\bmod\,I_M)$. Then there exists $w=w_1\dots w_m\in\langle
X\rangle_m$ such that $w_m>v_k$ and $u=w\,(\bmod\,I_M)$.
\end{lemma}

\begin{proof} If $u_m>v_k$, we can just take $w=u$. Thus we can
assume that $v_k\geq u_m$. By Lemma~\ref{subm}, $u_1\dots u_{m-1}$
is singular and therefore Lemma~\ref{sing} provides $y=y_1\dots
y_{m-1}\in \langle X\rangle_{m-1}$ such that $y=u_1\dots
u_{m-1}\,(\bmod\,I_M)$ and $y_{m-1}\geq v_k$. Clearly,
\begin{equation}\label{E1}
u=y_1\dots y_{m-2}y_{m-1}u_m\,(\bmod\, I_M).
\end{equation}
Since $y_{m-1}\geq v_k\geq u_m$, the monomial $y_{m-1}u_m$ features
in one of the elements of $M$. Next, $y_{m-1}\neq b$ (otherwise
(\ref{E1}) implies that $u$ is tame). Hence $y_{m-1}u_m\neq ba$ and
therefore $y_{m-1}u_m$ belongs to the support of some $g\in M'$.
Taking into account that $u_m$ is pure, we have three possibilities:
either $g=y_{m-1}u_m$ or $g=y_{m-1}u_m-cd$ with $d<u_{m}$ or
$g=cd-y_{m-1}u_m$ with $d>y_{m-1}$.

If $g=y_{m-1}u_m$, $y_{m-1}u_m\in I_M$ and therefore (\ref{E1})
implies that $u\in I_M$, which contradicts the minimality of $u$.
Hence $g=y_{m-1}u_m-cd$ with $d<u_{m}$ or $g=cd-y_{m-1}u_m$ with
$d>y_{m-1}$. In either case $y_{m-1}u_m=cd\,(\bmod\,I_M)$ and
(\ref{E1}) gives
\begin{equation}\label{E2}
u=y_1\dots y_{m-2}cd\,(\bmod\,I_M).
\end{equation}
In the case $d<u_{k+1}$ the monomial in the right-hand-side of
(\ref{E2}) is less than $u$, which contradicts the minimality of
$u$. It remains to consider the case $d>y_{m-1}\geq v_k$. In this
case $w=y_1\dots y_{m-2}cd$ satisfies all desired conditions.
\end{proof}

\begin{corollary}\label{npu} Let $X$ be an $n$-element totally
ordered set, $a=\min X$, $b=\max X$, $M$ be a QHS on $X$,
$M'=M\setminus\{ba\}$, $I_M$ be the ideal in $\K\langle X\rangle$
generated by $M'$ and $u=u_1\dots u_m\in\langle X\rangle_m$ be a
singular monomial. Then the set $\{j:\text{$u_j$ is pure}\}$ has at
most $n-1$ elements.
\end{corollary}

\begin{proof}Assume the contrary. Then we can choose
$j_1,\dots,j_n\in\N$ such that $1\leq j_1<{\dots}<j_n\leq n$ and
each $u_{j_k}$ is pure. Let $a_1=u_{j_1}$ and take $y^{(1)}=u_1\dots
u_{j_1-1}$. Clearly $u_1\dots u_{j_1}=y^{(1)}a_1$ and therefore
$u_1\dots u_{j_1}=y^{(1)}a_1\,(\bmod\,I_M)$. By Lemma~\ref{singp},
there are $a_2>a_1$ and a monomial $y^{(2)}\in \langle
X\rangle_{j_2-1}$ such that $u_1\dots
u_{j_2}=y^{(2)}a_2\,(\bmod\,I_M)$. Applying Lemma~\ref{singp} again,
we see that there are $a_3>a_2$ and a monomial $y^{(3)}\in \langle
X\rangle_{j_3-1}$ such that $u_1\dots
u_{j_3}=y^{(3)}a_3\,(\bmod\,I_M)$. Proceeding this way we obtain
$a_1,\dots,a_n\in X$ and the monomials $y^{(k)}\in \langle
X\rangle_{j_k-1}$ such that $u_1\dots
u_{j_k}=y^{(k)}a_k\,(\bmod\,I_M)$ for $1\leq k\leq n$ and, most
importantly, $a_1<{\dots}<a_n$. Since $X$ has $n$ elements, the
latter can only happen if $a_n=b$. Then the equality $u_1\dots
u_{j_n}=y^{(n)}a_n\,(\bmod\,I_M)$ implies that
$u=y^{(n)}bu_{j_n+1}\dots u_m\,(\bmod\,I_M)$, which means that $u$
is tame. This contradiction completes the proof.
\end{proof}

\section{The extension construction for a Quasierh\"ohungssysteme}

Throughout this section $n\geq 5$ is an integer,
$X=\{x_1,\dots,x_n\}$ is an $n$-element set with the total ordering
$x_1<x_2<{\dots}<x_n$ and the $(n-4)$-element set
$X_0=X\setminus\{x_1,x_2,x_{n-1},x_n\}=\{x_3,\dots,x_{n-2}\}$
carries the total ordering inherited from $X$. Let also $M_0$ be a
QHS on $X_0$. By Remark~\ref{rem1}, $x_{n-2}x_3\in
M_0$. As usual, $M'_0=M_0\setminus\{x_{n-2}x_3\}$.

We consider $M,M'\subset\K\langle X\rangle$ defined in the following
way:
\begin{align*}
M'&=M'_0\cup \{x_nx_j-x_jx_1:2\leq j\leq n-2\}
\cup\{x_{n-1}x_{j+1}-x_jx_2:2\leq j\leq n-3\}
\\
&\quad\cup\{x_nx_n-x_{n-1}x_1,\,x_nx_{n-1}-x_{n-2}x_2,\,x_{n-1}x_{n-1}-
x_{n-2}x_3,\,x_{n-1}x_2-x_1x_1\}
\end{align*}
and
$$
M=M'\cup\{x_nx_1\}.
$$
It is straightforward to verify that $M$ is a QHS on
$X$.

Symbols $J$ and $I$ stand for the ideals in $\K\langle X\rangle$
generated by $M$ and $M'$ respectively. Similarly, $J_0$ and $I_0$
are the ideals in $\K\langle X_0\rangle$ generated by $M_0$ and
$M'_0$ respectively.

\begin{lemma}\label{m2} Let $m\in\N$ and
$u=u_1u_2\dots u_m\in\langle X\rangle_m$ be a singular $($with
respect to $M)$ monomial. Then the set
$\bigl\{j:u_j\in\{x_2,x_n,x_{n-1}\}\bigr\}$ has at most $n-1$
elements.
\end{lemma}

\begin{proof}From the way we defined $M$ it easily follows that
$x_n$, $x_{n-1}$ and $x_2$ are pure. It remains to apply
Corollary~\ref{npu}.
\end{proof}

\begin{lemma}\label{m33} Let $k$ be a non-negative integer and
$a\in\{x_2,\dots,x_{n-2}\}$. Then $x_n^ka=ax_1^k\,(\bmod\,I)$.
\end{lemma}

\begin{proof}From the definition of $M$ it follows that
$x_na-ax_1\in M'\subset I$. That is, $x_na=ax_1\,(\bmod\,I)$.
The required inequality now follows via an obvious inductive argument.
\end{proof}

\begin{lemma}\label{m1} Let $q=\frac{5n-9}{2}$ if $n$ is odd and
$q=\frac{5n-8}{2}$ if $n$ is even. Then there exists a monomial
$u=u_1\dots u_q\in \langle X\rangle_q$ such that $u_q=x_n$ and
$u=x_1^q\,(\bmod\,I)$.
\end{lemma}

\begin{proof} Since $x_{n-1}x_2-x_1x_1\in M'\subset I$, we see that for every
$k\geq 2$, $x_1^k=x_{n-1}x_2x_1^{k-2}\,(\bmod\,I)$. By Lemma~\ref{m33}, $x_2x_1^{k-2}=x_n^{k-2}x_2\,(\bmod\,I)$. Hence
\begin{equation}\label{f11}
\text{$x_1^k=x_{n-1}x_n^{k-2}x_2\,(\bmod\,I)$ for every $k\geq 2$.}
\end{equation}
Since $x_{n}x_n-x_{n-1}x_1\in M'$ and $x_{n}x_{n-1}-x_{n-2}x_2\in
M'$, we get $x_{n}^3=x_nx_{n-1}x_1=x_{n-2}x_2x_1\,(\bmod\,I)$. Hence
$x_{n}^k=x_{n}^{k-3}x_{n-2}x_2x_1\,(\bmod\,I)$.  By Lemma~\ref{m33},
$x_{n}^{k-3}x_{n-2}=x_{n-2}x_1^{k-3}\,(\bmod\,I)$. Combining the
last two equalities, we obtain
\begin{equation}\label{f22}
\text{$x_{n}^k=x_{n-2}x_1^{k-3}x_2x_1\,(\bmod\,I)$ for every $k\geq 3$.}
\end{equation}
For $k\geq 5$, we can apply (\ref{f22}) to $x_n^{k-2}$ in (\ref{f11}), which gives
\begin{equation}\label{f33}
\text{$x_1^k=x_{n-1}x_{n-2}x_1^{k-5}x_2x_1x_2\,(\bmod\,I)$ for every $k\geq 5$.}
\end{equation}
Since $x_{n}x_2-x_{2}x_1\in M'$ and $x_{n-1}x_{3}-x_{2}x_2\in M'$,
we have
\begin{equation}\label{f44}
\text{$x_2x_1x_2=x_{n}x_{2}x_2=x_{n}x_{n-1}x_{3}\,(\bmod\,I)$.}
\end{equation}
Next, for $2\leq j\leq n-4$, we have $x_{n-1}x_{j+1}-x_{j}x_{2}\in
M'$, $x_{n}x_{j+1}-x_{j+1}x_{1}\in M'$ and
$x_{n-1}x_{j+2}-x_{j+1}x_{2}\in M'$. Hence,
\begin{equation}\label{f55}
\text{$x_jx_2x_1x_2=x_{n-1}x_{j+1}x_1x_2=x_{n-1}x_{n}x_{j+1}x_{2}=
x_{n-1}x_{n}x_{n-1}x_{j+2}\,(\bmod\,I)$ for $2\leq j\leq n-4$.}
\end{equation}
Finally, since $x_{n-1}x_{n-2}-x_{n-3}x_{2}\in M'$,
$x_{n}x_{n-2}-x_{n-2}x_{1}\in M'$, $x_{n}x_{n-1}-x_{n-2}x_{2}\in M'$
and $x_nx_n-x_{n-1}x_1\in M'$, we get
\begin{equation}\label{f66}
\begin{array}{l}
x_{n-3}x_2x_1x_2x_1=x_{n-1}x_{n-2}x_1x_2x_1=x_{n-1}x_{n}x_{n-2}x_{2}x_1=
\\
\qquad\qquad\qquad\,\,=x_{n-1}x_{n}x_{n}x_{n-1}x_1=x_{n-1}x_nx_nx_nx_n\,(\bmod\,I).
\end{array}
\end{equation}

{\bf Case 1}: $n$ is odd. In this case $n=2k+3$ and $q=5k+3$ for some $k\in\N$. Using (\ref{f33}) consecutively $k$ times, we see that
$$
x_1^q=(x_{n-1}x_{n-2})^kx_1x_1(x_2x_1x_2)^{k}x_1\,(\bmod\,I).
$$
Since $x_1x_1=x_{n-1}x_{2}\,(\bmod\,I)$, we obtain
$$
x_1^q=(x_{n-1}x_{n-2})^kx_{n-1}x_{2}(x_2x_1x_2)^{k}x_1\,(\bmod\,I).
$$
Applying (\ref{f55}) consecutively $k-1$ times, we have
$$
x_1^q=(x_{n-1}x_{n-2})^kx_{n-1}(x_{n-1}x_{n}x_{n-1})^{k-1}
x_{n-3}x_2x_1x_2x_1\,(\bmod\,I).
$$
According to (\ref{f66}),
$$
x_1^q=(x_{n-1}x_{n-2})^kx_{n-1}(x_{n-1}x_{n}x_{n-1})^{k-1}
x_{n-1}x_{n}x_{n}x_{n}x_n\,(\bmod\,I),
$$
which completes the proof in the case of odd $n$.

{\bf Case 2}: $n$ is even. In this case $n=2k+2$ and $q=5k+1$ for some $k\geq 2$. Using (\ref{f33}) consecutively $k$ times, we see that
$$
x_1^q=(x_{n-1}x_{n-2})^k(x_2x_1x_2)^{k}x_1\,(\bmod\,I).
$$
By (\ref{f44}),
$$
x_1^q=(x_{n-1}x_{n-2})^kx_{n}x_{n-1}x_{3}(x_2x_1x_2)^{k-1}x_1\,(\bmod\,I).
$$
Applying (\ref{f55}) consecutively $k-2$ times, we have
$$
x_1^q=(x_{n-1}x_{n-2})^kx_{n}x_{n-1}(x_{n-1}x_{n}x_{n-1})^{k-2}
x_{n-3}x_2x_1x_2x_1\,(\bmod\,I).
$$
According to (\ref{f66}),
$$
x_1^q=(x_{n-1}x_{n-2})^kx_{n}x_{n-1}(x_{n-1}x_{n}x_{n-1})^{k-2}
x_{n-1}x_{n}x_{n}x_{n}x_n\,(\bmod\,I),
$$
which completes the proof in the case of even $n$.
\end{proof}

\begin{lemma}\label{m30} Let $k$ be a non-negative integer. If $ab\in
M_0'$, then $ax_1^kb\in I$. If $ab-cd\in M_0'$, then
$ax_1^kb=cx_1^kd\,(\bmod\,I)$.
\end{lemma}

\begin{proof} If $ab\in M_0'$, we have $ab\in I_0\subset I$ and therefore
Lemma~\ref{m33} implies that $ax_1^kb=x_n^kab\in I$.

If $ab-cd\in M_0'$, we can use Lemma~\ref{m33} to see that
$ax_1^kb=x_n^kab=x_n^kcd=cx_1^kd\,(\bmod I)$.
\end{proof}

Since $I_0$ is the ideal in $\K\langle X_0\rangle$ generated by
$M_0'$, the above lemma immediately implies the following result.

\begin{corollary}\label{m3} Let $u,v\in \langle X_0\rangle_m$ and
$u=v\,(\bmod\,I_0)$. Then $\widetilde u=\widetilde v\,(\bmod\,I)$
for every non-negative integers $k_1,\dots,k_{m-1}$, where
$u=u_1\dots u_m$, $v=v_1\dots,v_m$, $\widetilde
u=u_1x_1^{k_1}u_2x_1^{k_2}\dots x_1^{k_{m-1}}u_{m}$ and $\widetilde
v=v_1x_1^{k_1}v_2x_1^{k_2}\dots x_1^{k_{m-1}}v_{m}$.
\end{corollary}

\begin{lemma}\label{m4} Let $m\in\N$, $m\geq 3$ and $u_1,\dots,u_{m}\in
\langle X_0\rangle_{m}$ be such that the monomial $u_1\dots u_{m-2}$
is $M_0$-tame. Let also $k_1,\dots,k_{m-1}$ be non-negative integers
and $w=u_1x_1^{k_1}u_2x_1^{k_2}\dots x_1^{k_{m-1}}u_{m}$. Then $w$ is
not $M$-singular.
\end{lemma}

\begin{proof} Since $u_1\dots u_{m-2}$
is $M_0$-tame, there are $v_1\dots v_{m-2}\in
\langle X_0\rangle_{m-2}$ and $j\in\{1,\dots,m-2\}$ such that
\begin{equation}\label{fuu}
\text{$u_1\dots u_{m-2}=v_1\dots v_{m-2}\,(\bmod\,I_0)$,
$v_j=x_{n-2}$ and $u_l=v_l$ for $j<l\leq m-2$.}
\end{equation}
According to (\ref{fuu}) and Corollary~\ref{m3},
\begin{equation*}%\label{fuu1}
w=v_1x_1^{k_1}\dots v_{j-1}x_1^{k_{j-1}}x_{n-2}x_1^{k_{j}}
u_{j+1}x_1^{k_{j+1}}\dots u_{m-1}x_1^{k_{m-1}}u_{m}\,(\bmod\,I).
\end{equation*}
Applying Lemma~\ref{m33} to $x_{n-2}x_1^{k_{j}}$ and using the above
display, we obtain
\begin{equation}\label{fuu1}
\text{$w=yx_{n-2}u_{j+1}x_1^{k_{j+1}}\dots u_{m-1}x_1^{k_{m-1}}u_{m}\,(\bmod\,I)$, where $y=v_1x_1^{k_1}\dots v_{j-1}x_1^{k_{j-1}}x_n^{k_{j}}$.}
\end{equation}

Since $x_{n-2}=\max X_0$, the monomial $x_{n-2}u_{j+1}$ features in
exactly one $g\in M_0$. If $u_{j+1}>x_3$, $g\in M'_0$ and there are
two possibilities: either $g=x_{n-2}u_{j+1}$ or
$g=x_{n-2}u_{j+1}-cd$ with $d<u_{j+1}$. In the first case
(\ref{fuu1}) implies that $w\in I$ and therefore $w$ is non-minimal
and therefore non-singular. In the second case (\ref{fuu1}) gives
$$
w=ycdx_1^{k_{j+1}}\dots u_{m-1}x_1^{k_{m-1}}u_{m}\,(\bmod\,I).
$$
The monomial in the right-hand side of the above display is less than $w$ and again we see that $w$ is non-minimal and therefore non-singular.

It remains to consider the case  $u_{j+1}=x_3$. Since $x_{n-1}x_{n-1}-x_{n-2}x_{3}\in M'$, we have $x_{n-2}x_{3}=x_{n-1}x_{n-1}\,(\bmod\,I)$. Thus in this case (\ref{fuu1}) implies that
\begin{equation}\label{fuu2}
w=yx_{n-1}x_{n-1}x_1^{k_{j+1}}\dots u_{m-1}x_1^{k_{m-1}}u_{m}\,(\bmod\,I).
\end{equation}
Again, we have two possibilities: either $k_{j+1}>0$ or $k_{j+1}=0$. First, assume that $k_{j+1}=0$. In this case (\ref{fuu2}) reads
$$
w=yx_{n-1}x_{n-1}u_{j+2}x_1^{k_{j+2}}\dots u_{m-1}x_1^{k_{m-1}}u_{m}\,(\bmod\,I).
$$
Since $u_{j+2}\in X_0$, from the definition of $M$ it follows that there is
$g\in M'$ such that $g=x_{n-1}u_{j+2}-cd$ with $d<u_{j+1}$. Then $x_{n-1}u_{j+2}=cd\,(\bmod\,I)$ and according to the above display
$$
w=yx_{n-1}cdx_1^{k_{j+2}}\dots u_{m-1}x_1^{k_{m-1}}u_{m}\,(\bmod\,I).
$$
The monomial in the right-hand side of the above display is less than $w$ and therefore $w$ is non-minimal and hence non-singular. Finally, assume that $k_{j+1}>0$. Since
$x_{n}x_{n}-x_{n-1}x_{1}\in M'$, we have $x_{n-1}x_{1}=x_{n}x_{n}\,(\bmod\,I)$ and (\ref{fuu2}) implies that
$$
w=yx_{n-1}x_{n}x_{n}x_1^{k_{j+1}-1}u_{j+3}x_1^{k_{j+3}}\dots u_{m-1}x_1^{k_{m-1}}u_{m}\,(\bmod\,I).
$$
The above display shows that $w$ is tame if it is minimal. Thus in any case, $w$ is non-singular.
\end{proof}

The main result of this section is the following theorem.

\begin{theorem}\label{ext} If $M_0$ is regular, then $M$ is
also regular.
\end{theorem}

\begin{proof} Assume that $M_0$ is regular. Then there is $m\in\N$ such that
there are no $M_0$-singular monomials in $\langle X_0\rangle_{m-2}$.
Let $q=\frac{5n-9}{2}$ if $n$ is odd and $q=\frac{5n-8}{2}$ if $n$
is even and set $r=nmq$.

It suffices to show that there are no $M$-singular monomials in
$\langle X\rangle_r$. Assume that there exists an $M$-singular $u\in
\langle X\rangle_r$. By Lemma~\ref{m2}, the set
$\bigl\{j:u_j\in\{x_2,x_n,x_{n-1}\}\bigr\}$ has at most $n-1$
elements. It follows that there is a submonomial $v$ of $u$ of
degree $mq$ such that $v$ does not contain $x_2$, $x_{n-1}$ or
$x_n$. By Lemma~\ref{subm}, $v$ is singular. If at least $m$ entries
of $v$ belong to $X_0$, Lemma~\ref{m4} implies that $v$ is
non-singular and we arrive to a contradiction. Thus the degree $mq$
monomial $v$ contains at most $m-1$ entries from $X_0$. Since all
other entries of $v$ are $x_1$, it follows that $x_1^q$ is a
submonomial of $v$. Hence $x_1^q$ is singular, which contradicts
Lemma~\ref{m1}.
\end{proof}

It is straightforward to see that the cardinalities of $M$ and $M_0$
are related by the equality $|M|=|M_0|+2n-3$. Hence
Theorem~\ref{ext} immediately implies the following result.

\begin{corollary}\label{ext1} Suppose that $n\geq 5$ and that there is a $k$-element regular QHS on an $(n-4)$-element set. Then there is a $(k+2n-3)$-element
regular QHS on an $n$-element set.
\end{corollary}

Applying Corollary~\ref{ext1} several times, we, after an exercise
on summing up an arithmetic series, obtain the following result.

\begin{proposition}\label{ext2} Let $m,j\in\N$ and there is a $k$-element
regular QHS on an $m$-element set. Then there is a $(k+2jm+4j^2+j)$-element
regular QHS on an $(m+4j)$-element set.
\end{proposition}

\section{Proof of Theorem~\ref{semi2}}

We start with the following 4 specific examples of QHS. For $1\leq m\leq 4$, we set $X_m=\{x_1,\dots,x_m\}$ ordered by $x_1<x_2<x_3<x_4$. Now we define a QHS $M_m$ on the $m$-element totally ordered set $X_m$:
$$
\begin{array}{l}
M_1=\{x_1x_1\},\ \ \ M_2=\{x_2x_2-x_1x_1,x_2x_1\},\ \ \ M_3=\{x_3x_3-x_2x_1, x_3x_2-x_1x_1, x_3x_1, x_2x_2\},
\\
M_4=\{x_4x_4-x_3x_1,
x_4x_3-x_2x_1, x_4x_2-x_1x_1, x_3x_3-x_2x_2, x_3x_2, x_4x_1\}.
\end{array}
$$

\begin{lemma}\label{small}
For each $m\in\{1,2,3,4\}$, the QHS $M_m$ is regular.
\end{lemma}

\begin{proof} It is straightforward to see that for each $M_m$, each
element of $X_m$ is pure. By Corollary~\ref{npu} an $M_m$-singular
monomial can not have the degree more than $m-1$. Hence there are no
$M_m$-singular monomials of degree $m$. By definition, each $M_m$ is
regular.
\end{proof}

Let $j$  be an arbitrary non-negative integer. By Lemma~\ref{small}, there is a regular $1$-element QHS on a 1-element set. By Proposition~\ref{ext2}, there is a $(4j^2+3j+1)$-element regular QHS on a $(4j+1)$-element set.
By Lemma~\ref{small}, there is a regular $2$-element QHS on a
2-element set. By Proposition~\ref{ext2}, there is a $(4j^2+5j+2)$-element regular QHS on a $(4j+2)$-element set. By Lemma~\ref{small}, there is a regular $4$-element QHS on a 3-element set. By Proposition~\ref{ext2}, there is a $(4j^2+7j+4)$-element regular QHS on a $(4j+3)$-element set. Finally, by Lemma~\ref{small}, there is a regular $6$-element QHS on a 4-element set. By Proposition~\ref{ext2}, there is a $(4j^2+9j+6)$-element regular QHS on a $(4j+4)$-element set.

For $n\in\N$, we define $\delta_n=4j^2+3j+1$ if $n=4j+1$,
$\delta_n=4j^2+5j+2$ if $n=4j+2$, $\delta_n=4j^2+7j+4$ if $n=4j+3$
and $\delta_n=4j^2+9j+6$ if $n=4j+4$. The above observations mean
that for every $n\in\N$, there is a $\delta_n$-element regular QHS
on an $n$-element set. According to Lemma~\ref{regqhs}, this means
that for every $n\in\N$, there is a quadratic semigroup algebra
$R_n$ with $n$ generators given by $\delta_n$ relations such that
$R_n$ is finite dimensional. On the other hand, a routine check
shows that $\delta_n$ is precisely the smallest integer greater than
$\frac{n^2+n}{4}$. This completes the proof of Theorem~\ref{semi2}.

\bigskip

We conclude with the following remarks.

\begin{remark}\label{pu} We have just seen that for $n\leq 4$, there
is a $\delta_n$-element QHS on an $n$-element set $X$ for which
every $c\in X$ is pure. It turns out that $n=4$ is the last $n$ for
which this phenomenon occurs. More precisely, one can easily see
that a QHS on an $n$-element set $X$, for which every $c\in X$ is
pure, must contain at least $\frac{n^2+2n}{4}$ elements.
\end{remark}

\begin{remark}\label{ver} Vershik \cite{vers} conjectured that for
every $n\geq 3$, there is a finite dimensional quadratic algebra
with $n$ generators given by $\frac{n^2-n}{2}$ relations. This
conjecture is proved in \cite{na} (see also \cite{cana} for the
proof in the case $3\leq n\leq 7$). Note that for
$n\geq 4$, this statement could also  be obtained as a consequence of Theorem~\ref{semi2}.
\end{remark}

\begin{remark}\label{ttt} It would be interesting to get a
tight enough estimate of the order of nilpotency of quadratic
semigroup algebras which we construct while proving Theorem~\ref{semi2}.
An estimate, that follows from the proof is way higher than the
actual value for small $n$.
\end{remark}

\begin{remark}\label{lastt} Let $R$ be a quadratic algebra with $n$ generators given by  $d\leq n^2$ relations. For $k\in\N$, $R_k$ stands for the degree $k$ homogeneous component of $R$. By the Golod--Shafarevich estimate $\dim R_3\geq\min\{0,n^3-2dn\}$. Anick \cite{ani1,ani2} proved that this estimate for $\dim R_3$ is actually attained for every $n$ and $d$. An analysis of his proof shows that it is also attained if we restrict ourselves to semigroup algebras.

On the other hand (see \cite{na}), there is an $R$ with $n=d=3$ for which $R_5=\{0\}$, while Theorem~\ref{semi1} implies that $\dim R_5>0$ if $R$ is a semigroup algebra with $n=d=3$. So, for every pair $(n,d)$, the minimal dimensions of the third component for general quadratic algebras and for semigroup quadratic algebras coincide, while the same statement fails for the dimension of the fifth component. The situation with the fourth component remains unclear. Its clarification is a part of the following much more general question.
\end{remark}

\begin{question}\label{zzz} What is the minimal dimension of  $R_k$, where $R$ is a semigroup quadratic algebra with $n$ generators given by $d\leq n^2$ relations?
\end{question}

\bigskip {\bf Acknowledgements.} \ We are grateful to the referee for a comment, which prompted us to improve the presentation of the paper.

\small\rm

\normalsize

\vskip1truecm

\scshape

\noindent  Natalia Iyudu\ \ {\rm and}\ \  Stanislav Shkarin \\

\noindent Max-Planck-Institut f\"ur Mathematik

\noindent 7 Vivatsgasse, 53111 Bonn

\noindent Germany

\noindent E-mail address: \qquad {\tt iyudu@mpim-bonn.mpg.de}\\

\noindent Queens's University Belfast

\noindent Department of Pure Mathematics

\noindent University road, Belfast, BT7 1NN, UK

\noindent E-mail address: \qquad {\tt s.shkarin@qub.ac.uk}

\end{document}